\documentclass[12pt]{amsart}
\usepackage{amscd,amsmath,amsthm,amssymb, xcolor}
\def\NZQ{\mathbb}               % the font for N,Z,Q,R,C

\def\ZZ{{\NZQ Z}}

\def\frk{\mathfrak}               % font for "Fraktur"

\def\mm{{\frk m}}

\def\G{{\mathcal G}}

\def\opn#1#2{\def#1{\operatorname{#2}}} % to make operators
\opn\chara{char} \opn\length{\ell} \opn\pd{pd} \opn\rk{rk}
\opn\projdim{proj\,dim} \opn\injdim{inj\,dim} \opn\rank{rank}
\opn\depth{depth} \opn\grade{grade} \opn\height{height}
\opn\embdim{emb\,dim} \opn\codim{codim}

\opn\Tr{Tr} \opn\bigrank{big\,rank}
\opn\superheight{superheight}\opn\lcm{lcm}
\opn\trdeg{tr\,deg}%\emph{
\opn\reg{reg} \opn\lreg{lreg} \opn\ini{in} \opn\lpd{lpd}
\opn\size{size} \opn\sdepth{sdepth}
\opn\link{link}\opn\fdepth{fdepth}\opn\lex{lex}
\opn\tr{tr}
\opn\type{type}
\opn\gap{gap}
\opn\arithdeg{arith-deg}
\opn\astab{astab}
\opn\dstab{dstab}
\opn\bigheight{bigheight}
\opn\facets{Facets}
\opn\div{div} \opn\Div{Div} \opn\cl{cl} \opn\Cl{Cl}
\opn\Spec{Spec} \opn\Supp{Supp} \opn\supp{supp} \opn\Sing{Sing}
\opn\Ass{Ass} \opn\Min{Min}\opn\Mon{Mon}
\opn\Ann{Ann} \opn\Rad{Rad} \opn\Soc{Soc}
\opn\Im{Im} \opn\Ker{Ker} \opn\Coker{Coker} \opn\Am{Am}
\opn\Hom{Hom} \opn\Tor{Tor} \opn\Ext{Ext} \opn\End{End}
\opn\Aut{Aut} \opn\id{id}

\opn\nat{nat}
\opn\pff{pf}%   \pf exists already
\opn\Pf{Pf} \opn\GL{GL} \opn\SL{SL} \opn\mod{mod} \opn\ord{ord}
\opn\Gin{Gin} \opn\Hilb{Hilb}\opn\sort{sort}
\opn\PF{PF}\opn\Ap{Ap}
\opn\mult{mult}
\opn\aff{aff}
\opn\relint{relint} \opn\st{st}
\opn\lk{lk} \opn\cn{cn} \opn\core{core} \opn\vol{vol}  \opn\inp{inp} \opn\nilpot{nilpot}
\opn\link{link} \opn\star{star}\opn\lex{lex}\opn\set{set}
\opn\width{wd}
\opn\Fr{F}
\opn\QF{QF}
\opn\G{G}
\opn\type{type}\opn\res{res}
\opn\conv{conv}
\opn\gr{gr}

\def\pot#1#2{#1[\kern-0.28ex[#2]\kern-0.28ex]}

\opn\dirlim{\underrightarrow{\lim}}
\opn\inivlim{\underleftarrow{\lim}}
\let\union=\cup
\let\sect=\cap

\let\to=\rightarrow

\def\Implies{\ifmmode\Longrightarrow \else
	\unskip${}\Longrightarrow{}$\ignorespaces\fi}
\def\implies{\ifmmode\Rightarrow \else
	\unskip${}\Rightarrow{}$\ignorespaces\fi}
\def\iff{\ifmmode\Longleftrightarrow \else
	\unskip${}\Longleftrightarrow{}$\ignorespaces\fi}

\let\:=\colon
\newtheorem{Theorem}{Theorem}[section]
\newtheorem{Lemma}[Theorem]{Lemma}
\newtheorem{Corollary}[Theorem]{Corollary}
\newtheorem{Proposition}[Theorem]{Proposition}

\newtheorem{Example}[Theorem]{Example}
\newtheorem{Examples}[Theorem]{Examples}

\usepackage{amsthm}
\newtheorem*{conjecture}{Conjecture}
\newtheorem*{corollary}{Corollary}
\newtheorem*{theorem}{Theorem}

\newtheorem*{problem}{Problem}

\let\epsilon\varepsilon
\let\kappa=\varkappa
\def\qed{\ifhmode\textqed\fi
	\ifmmode\ifinner\quad\qedsymbol\else\dispqed\fi\fi}
\def\textqed{\unskip\nobreak\penalty50
	\hskip2em\hbox{}\nobreak\hfil\qedsymbol
	\parfillskip=0pt \finalhyphendemerits=0}
\def\dispqed{\rlap{\qquad\qedsymbol}}

\opn\dis{dis}
\def\pnt{{\raise0.5mm\hbox{\large\bf.}}}

\opn\Lex{Lex}

\begin{document}
	
	\title {The normalized depth function  of squarefree powers}
	
	\author {Nursel Erey, J\"urgen Herzog, Takayuki Hibi  and Sara Saeedi Madani}
	
	\address{Nursel Erey, Gebze Technical University, Department of Mathematics, 41400 Gebze, Kocaeli, Turkey}
	\email{nurselerey@gtu.edu.tr}

	\address{J\"urgen Herzog, Fachbereich Mathematik, Universit\"at Duisburg-Essen, Campus Essen, 45117
		Essen, Germany}
	\email{juergen.herzog@uni-essen.de}

	\address{Takayuki Hibi, Department of Pure and Applied Mathematics, Graduate School of Information Science and Technology, Osaka University, Suita, Osaka 565--0871, Japan}
	\email{hibi@math.sci.osaka-u.ac.jp}

	\address{Sara Saeedi Madani, Department of Mathematics and Computer Science, Amirkabir University of Technology (Tehran Polytechnic), Iran, and School of Mathematics, Institute for Research in Fundamental Sciences (IPM), Tehran, Iran}
	\email{sarasaeedi@aut.ac.ir}

	\address{}
	
	\address{}

	\dedicatory{ }
	
	\begin{abstract}
		The depth of squarefree powers of a squarefree monomial ideal is introduced.  Let $I$ be a squarefree monomial ideal of the polynomial ring $S=K[x_1,\ldots,x_n]$.  The $k$-th squarefree power $I^{[k]}$ of $I$ is the ideal of $S$ generated by those squarefree monomials $u_1\cdots u_k$ with each $u_i\in G(I)$, where $G(I)$ is the unique minimal system of monomial generators of $I$.  Let $d_k$ denote the minimum degree of monomials belonging to $G(I^{[k]})$.  One has $\depth(S/I^{[k]}) \geq d_k -1$.  Setting $g_I(k) = \depth(S/I^{[k]}) - (d_k - 1)$, one calls $g_I(k)$ the normalized depth function of $I$.  The computational experience strongly invites us to propose the conjecture that the normalized depth function is nonincreasing.  In the present paper, especially the normalized depth function of the edge ideal of a finite simple graph is deeply studied.
		
		%The depth of squarefree powers of a squarefree monomial ideal is studied.
	\end{abstract}
	
	\thanks{}
	
	%		13H10   	Special types (Cohen-Macaulay, Gorenstein, Buchsbaum, etc.)
	%		13D02   	Syzygies, resolutions, complexes
	%		05E40   	Combinatorial aspects of commutative algebra
	%		16S36   	Ordinary and skew polynomial rings and semigroup rings
	
	%		14M25   	Toric varieties, Newton polyhedra [See also 52B20]
	%		13A02   	Graded rings
	%		13F20   	Polynomial rings and ideals; rings of integer-valued polynomials
	%		13A18   	Valuations and their generalizations
	%		06A11   	Algebraic aspects of posets
	
	
	\maketitle
	
	\setcounter{tocdepth}{1}
	%\tableofcontents
	
	\section*{Introduction}
Let $K$ be a field,  and let $S=K[x_1,\ldots,x_n]$ be the polynomial ring in $n$ indeterminates over $K$.  The depth function of a homogeneous ideal $I$ is the integer valued function $f_I(k)=\depth (S/I^k)$.  While it is known by Brodmann \cite{B} that $f_I(k)$ is constant for all $k\gg 0$,  the initial behaviour of the depth function is not so easy to understand. In \cite{HH}  it was conjectured that any bounded convergent function $\ZZ_{\geq 0}\to \ZZ_{\geq 0}$ could be the depth function of a suitable ideal.  This conjecture has been   proved several years later  by H.T.~H\`{a}, H.~Nguyen, N.~Trung and T.~Trung in \cite[Theorem 4.1]{HNTT}.  

For a longer time it was expected that the depth function of a squarefree  monomial ideal is nonincreasing. Francisco, H\`{a} and Van Tuyl \cite{FHT}  showed that this expected behaviour for the powers of unmixed height~2 squarefree monomial ideals  would be a consequence of a combinatorial statement which says that for every positive integer $k$ and every $k$-critical (i.e., critically
$k$-chromatic) graph, there is a set of vertices whose replication
produces a $(k + 1)$-critical graph. However in 2014,  Kaiser, Stehl\'{i}k and \v{S}rekovski \cite{KSS} gave a counterexample to this and constructed an example  of a squarefree  monomial ideal  $I\subset S$  with $\depth (S/I^3) =0$ but $\depth(S/I^4)=4$.  It is still  open whether the depth function of the edge ideal of a graph is nonincreasing.

In the present paper, we study squarefree monomial ideals and their squarefree powers. Several algebraic properties of such powers have been studied in \cite{BHZ}, \cite{EHHS} and \cite{EH}.
Let  $I\subset S$ be a squarefree monomial ideal. The uniquely determined  minimal set of generators of $I$ is denoted by $G(I)$.  We denote by $I^{[k]}$ the $k$th \emph{squarefree power} of $I$. The generators of $I^{[k]}$ are the  products $u_1\cdots u_k$ with $u_i \in G(I)$, which form a squarefree monomial. Thus $u_1\cdots u_k\in G(I^{[k]})$ if and only if $u_1,\ldots, u_k$ is a regular sequence. 

A case of special interest is the squarefree powers of the edge ideal of a graph. Let $G$ be a finite simple graph on $[n] = \{1,\ldots,n\}$,  and let as before $S=K[x_1, \ldots, x_n]$ be the polynomial ring in $n$ variables over a field $K$.  The \emph{edge ideal} of $G$ is the squarefree monomial ideal $I(G)$ of $S$,  which is generated by those $x_ix_j$ for which $\{i, j\}$ is an edge of $G$.  Now, given an integer $k > 0$, the $k$th squarefree power of $I(G)$ is the squarefree monomial ideal $I(G)^{[k]}$ of $S$ which is generated by the  squarefree monomials
	\[
	x_{i_1}x_{j_1} x_{i_2}x_{j_2} \cdots x_{i_k}x_{j_k},
	\]	
where each $\{i_q, j_q\}$ is an edge of $G$ and where $\{i_q, j_q\} \cap \{i_r, j_r\} = \emptyset$ for $q \neq r$.  It follows that the minimal set of  generators of $I(G)^{[k]}$
are in bijection to the vertex sets of $k$-matchings of $G$. Recall that a set $M$ of edges of $G$ is called a \emph{matching}, if no distinct two edges of $G$ have a common vertex. The matching $M$ is called a \emph{$k$-matching}, if $|M|=k$. The set of all matchings form a simplicial complex, the so-called matching complex of $G$. The \emph{matching number} $\nu(G)$ of $G$ is the maximum cardinality of a matching of $G$. We have $I(G)^{[k]}\neq 0$ if and only if  $k>\nu(G)$.

Let again $I\subset S$ be an arbitrary squarefree monomial ideal.  We  denote by  $\nu(I)$ the maximum length of a monomial regular sequence in $I$. Thus if $G$ is a graph, then 
$\nu(I(G))=\nu(G)$.  We are interested in the depth of the squarefree powers of $I$. Let $d_k=\min\{\deg u\:  u \in G(I^{[k]})\}$. We also set $d=d_1$. Note that $d_{k+1}\geq d_k+d\geq (k+1)d$ for all $k<\nu(I)$.  Our first result is Proposition~\ref{prop:lower bound},  where it is shown that   $\depth (S/I^{[k]})\geq d_k-1$ for $k=1,\ldots,\nu(I)$.  In particular, in the important special case that $I$ is generated in  the single degree $d$,  we have  $\depth (S/I^{[k]})\geq dk-1$ for $k=1,\ldots,\nu(I)$. Thus, in contrast to ordinary powers, the depth function of squarefree powers  tends to be nondecreasing. The picture changes,  if we consider the function \[g_I(k)=\depth (S/I^{[k]})-(d_k-1)\]  for $k=1,\ldots,\nu(I)$. We call $g_I(k)$ the   {\em normalized depth function} of $I$. In all our results and the examples we considered,  $g_I(k)$ is a nonincreasing function. This fact and also Corollary~\ref{nu} leads us to the following 
\begin{conjecture}
\label{conjecture}
Let $I$ be a squarefree monomial ideal. Then  $g_I(k)$ is a nonincreasing function. 
\end{conjecture}
We say that the squarefree powers of $I$ have \emph{minimum depth} if  $g_I(k)=0$ for all $k=1,\dots ,\nu(I)$. In Section 1, we give  examples of ideals whose squarefree powers have minimum depth. Among them  are the edge ideals of complete graphs  and complete bipartite graphs. More generally,  any squarefree Veronese ideal as well as any matroidal ideal has minimum depth, see Examples~\ref{ramsau} and Theorem~\ref{matroid}. 

In the following two sections, we focus on edge ideals and  give criteria for minimum depth. One of the main results of this paper is 

\begin{corollary}
	Let $G$ be a graph with no isolated vertices and matching number $\nu(G)$. Then the following statements are equivalent:
	\begin{enumerate}
		\item[(i)] $G^c$ is disconnected.
		\item[(ii)] $g_{I(G)}(1)=0$.
		\item[(iii)] $g_{I(G)}(k)=0$ for all $1\leq k\leq \nu(G)$.
	\end{enumerate}
\end{corollary}

For the  proof of this result, the concept of well-ordered facet covers,  due to  Erey and Faridi \cite{EF},  is used  to give a non-vanishingness condition for Betti numbers of squarefree powers.

A subset $D$ of vertices of a graph $G$ is called a \emph{dominating set} if every vertex of $G$ which is not in $D$ is adjacent to some vertex in $D$. A complete subgraph $K_m$ of $G$ is called a \emph{dominating clique} if $V(K_m)$ is a dominating set. This  notion of dominating cliques  provides a sufficient condition for having minimum depth for a given squarefree power~$k$.  Indeed, we have 

\begin{theorem}\label{thm: dominating clique}
	Let $G$ be a graph and let $2\leq k\leq \nu(G)$. If $G$ has a dominating clique $K_{2k-1}$, then $g_{I(G)}(k)=0$.
\end{theorem}

We call a $k$-matching $M$ a \emph{dominating $k$-matching} if $V(M)$ is a dominating set.  For edge ideals with the property that $I^{[k]}$ has linear quotients, we have a criterion for minimum depth,  see Proposition~\ref{cochordal mindepth}.  As corollaries we obtain 

    \begin{corollary}\label{necessary condition}
	Let $G$ be a graph with no isolated vertices and $k$ be an integer with $1\leq k\leq \nu(G)$ where $I(G)^{[k]}$ has linear quotients. If $g_{I(G)}(k)=0$, then $G$ has a dominating $k$-matching. In particular, the statement holds for any cochordal graph $G$ with no isolated vertices.  	
\end{corollary} 

\begin{corollary}\label{nu}
	Let $G$ be a graph with no isolated vertices. Then $g_{I(G)}(\nu(G))=0$.
\end{corollary}

As a final application for our minimum depth criterion, we discuss the depth of squarefree powers of edge ideals of \emph{multiple whiskered complete graphs} which are obtained by attaching at least one whisker to each vertex of a complete graph $K_s$ with $s\geq 2$. We denote such a graph by $G=H(a_1,\ldots,a_s)$ where $a_i\geq 0$ is the number of whiskers attached to the vertex $i$ of $K_s$. For this family of graphs minimum depth is achieved in the second half of the interval $[1,\nu(G)]$. The precise result is the following.

\begin{theorem}\label{multiwhiskered depth}
	Let $G=H(a_1,\ldots,a_s)$ with $a_i\geq 1$ for all $i=1,\ldots,s$ and let  $k=1,\ldots,\nu(G)$. Then we have:
	\begin{enumerate}
		\item[(a)] $\nu(G)=s$;
		\item[(b)] $G$ is cochordal;
		\item[(c)] The following statements are equivalent: 
		\begin{enumerate}
			\item[(i)] $g_{I(G)}(k)=0$.
			\item[(ii)] $\lfloor s/2 \rfloor+1 \leq k \leq s$. 
		\end{enumerate}
	\end{enumerate}
\end{theorem}

Without any doubt,
one of the most challenging open questions is  to find  all  possible  normalized  depth functions. Our conjecture implies that if $g_I(k)=0$, then $g_I(k+1)=0$
for all $k<\nu(I)$. Therefore, as a partial answer to the above question, it would be nice to solve the following

\begin{problem}
For given integers $1 \leq s < m$, find a finite simple graph $G$ with $\nu(G) = m$ and 
		\begin{enumerate}
			\item[(i)] $g_{I(G)}(k)>0$ for $k = 1, \ldots, s$;
			\item[(ii)] $g_{I(G)}(k)=0$ for $k = s + 1, \ldots, m$.
		\end{enumerate}
\end{problem}

Throughout the paper, unless otherwise stated, $S=K[x_1,\ldots,x_n]$ is the polynomial ring in $n$ variables over a field $K$, all graphs and simplicial complexes have $n$ vertices corresponding to the $n$ variables of $S$.

\section{A lower bound for the depth of squarefree powers}\label{Sec:general squarefree powres}

%We fix a field $K$ and let $S=K[x_1,\ldots,x_n]$ be the polynomial ring in $n$ variables over $K$. 
%In contrast to the behaviour of the  depth of ordinary powers of ideals,  the depth of  squarefree powers of squarefree monomial ideals tend to grow with increasing powers. 

In this section, we consider squarefree powers of any squarefree monomial ideal and provide a lower bound for their depth. Besides discussing several examples, we also show that the normalized depth function of a matroidal ideal is zero. 
 
For the lower bound of the depth of squarefree powers we have the following result.

\begin{Proposition}\label{prop:lower bound}
Let $I\subset S$ be a squarefree monomial ideal. Then 
\begin{enumerate}
\item[(a)] $I^{[k]}=0$ if and only if $k> \nu(I)$;
\item[(b)] $g_I(k)\geq 0$ 
for all $k=1,\ldots, \nu(I)$.    
\end{enumerate}
\end{Proposition}
\begin{proof}
(a) Let $m=\nu(I)$ and $v_1,\ldots, v_m$ be a maximal regular sequence of monomials in $I$. Then for each $i$, there exists $u_i\in G(I)$ which divides $v_i$, and hence $u_1,\ldots,u_m$ is  again a maximal regular sequence of monomials in $I$. In particular, $\gcd(u_i,u_j)=1$ for  all $i\neq j$. It follows that $u_1\cdots u_m$ is squarefree. This implies that $I^{[k]}\neq 0$ for any $k\leq \nu(I)$.  Any product of generators with more than $m$  many factors cannot be squarefree, since these factors cannot form a regular sequence. This shows that $I^{[k]}= 0$ for $k>m$.

(b) Let $k=1,\ldots, \nu(I)$. By using the Auslander-Buchsbaum formula, it suffices to show that  $\projdim (S/I^{[k]})\leq n-d_k+1$ where $d_k=\min\{\deg u\:  u \in G(I^{[k]})\}$. We observe that for any $i>0$ for which $\beta_{i,j}(S/I^{[k]})\neq 0$,  we have $j\geq d_k+i-1$. 
	
It follows from Hochster's formula that $n\geq j$ for any $j$ such that $\beta_{i,j}(S/I^{[k]})\neq 0$. Therefore, $n\geq  d_k+i-1$, as desired. 
\end{proof}

Let $I$ be a squarefree monomial ideal, and let $k$ be an integer with $1\leq k\leq \nu(I)$. We say that $S/I^{[k]}$ has  {\em minimum depth} (or simply say that $I^{[k]}$ has \emph{minimum depth} when the polynomial ring is clear from the context) if $g_I(k)=0$.     

\begin{Examples} \label{ramsau} {\em
(a) Let $\mm$ be the graded maximal ideal of $S$, and let $I=\mm^{[d]}$ for some $d\leq n$. By \cite[Corollary~3.4]{HH}, we have $\depth(S/I)= d-1$. Since $I^{[k]}=\mm^{[kd]}$, it follows that $\depth(S/I^{[k]})= dk-1$ for $dk\leq n$. 

(b) Consider the polynomial rings  $S_1=K[x_1,\ldots,x_n]$ and  $S_2=K[y_1,\ldots,y_m]$, and let $I\subset S_1$ and  $J\subset S_2$ be graded ideals. Moreover, let $S=K[x_1,\ldots,x_n, y_1,\ldots,y_m]$. Then \cite[Corollary 3.2]{HRR} implies that 
$\depth(S/IJ)= \depth(S_1/I)+\depth(S_2/J)+1$.

In the given  situation,  assume that $I$ is a monomial ideal generated in degree $d_1$ and  $J$ is a monomial ideal generated in degree $d_2$.  Then $IJ$ is generated in degree $d=d_1+d_2$. We have $(IJ)^{[k]}=I^{[k]}J^{[k]}$, since $I$ and $J$ are ideals in different sets of variables. Therefore, $\depth(S/(IJ)^{[k]})= \depth(S_1/I^{[k]})+\depth(S_2/J^{[k]})-1$. By Proposition~\ref{prop:lower bound} we have 
$\depth(S/(IJ)^{[k]})\geq dk-1$, $\depth(S_1/I^{[k]})\geq d_1k-1$ and  $\depth(S_2/J^{[k]})\geq d_2k-1$. Thus we see that $S/(IJ)^{[k]}$ has minimum depth if and only if both  $S_1/I^{[k]}$ and $S_2/J^{[k]}$  have minimum depth. 

(c) Let $I$ be the edge ideal of a complete bipartite graph with the vertex set partition $[m]\union[n]$. Then $I=(x_1,\ldots,x_m)(y_1,\ldots,y_n)$, and we may apply (a) and (b) to see that $I^{[k]}$  has minimum depth for all $1\leq k\leq \nu(I)$. 

(d)   Let $G$ be the graph which is a 3-cycle with two whiskers at each vertex, and let $I$ be its edge ideal. It can be checked that $\depth(S/I)=5$ and $\depth(S/I^{[2]}) =3$. This shows that if $I\subset S$ is a squarefree monomial ideal generated in a single degree, then   $\depth(S/I^{[k]})$ is not necessarily an increasing function of $k$. This example, and a related family of graphs will be studied in Section~\ref{Sec:Linear quotients and minimal depth} in more details. 
}
\end{Examples}

 By the \emph{squarefree part} of a monomial ideal $J$, we mean the ideal generated by squarefree generators of $J$. It is clear that for any $k$, the squarefree part of $J^k$ coincides with $J^{[k]}$. 
 
 Let $I\subset S$ be a squarefree monomial ideal generated in degree $d$,  and suppose that $I^{[k]}=\mm^{[dk]}$ for some $k$ with $dk\leq n$. Then not only $S/I^{[k]}$ has minimum depth, but we also have $I^{[\ell]}=\mm^{[d\ell]}$ for all $\ell\geq k$ (which then implies that  $S/I^{[\ell]}$  has minimum depth for $\ell\geq k$ with $\ell k\leq n$).  This follows from Example~\ref{ramsau}~(a) and the next slightly more general result.

\begin{Proposition}
\label{powers equal}
Let $I\subset J\subset S$ be squarefree monomial ideals, and suppose that $I^{[k]}=J^{[k]}$ for some $k$. Then $I^{[\ell]}=J^{[\ell]}$ for all $\ell\geq k$.
\end{Proposition}

\begin{proof}
It suffices to show that $I^{[k+1]}=J^{[k+1]}$. For any two monomial ideals $L$ and $M$ we define the squarefree product, denoted by $L*M$,  as the squarefree part of $LM$. Since  $I^{[k]}\subseteq I^{[k-1]}*J\subseteq J^{[k]}=I^{[k]}$, it follows that $I^{[k]}= I^{[k-1]}*J$. Then 
\[
I^{[k+1]}=I*I^{[k]}=I*(I^{[k-1]}*J)=(I*I^{[k-1]})*J=I^{[k]}*J=J^{[k]}*J=J^{[k+1]}.
\]
\end{proof}

\begin{Example}\label{office}{\em
For $n\geq 4$,  let $P_n$ be the path graph with the vertex set $[n]=\{1,\ldots,n\}$ and edges $\{i,i+1\}$ for $i=1,\ldots,n-1$.  Let $G=P_n^c$ be the complementary graph of $P_n$ with the edge ideal $I$. By Corollary~\ref{nice},   
$\depth(S/I)>1$. In fact, we will show that $\depth(S/I)=2$. It suffices to show that  $\projdim(S/I)\geq n-2$. Indeed, $\{1,\ldots,n-2\}$ is a minimal vertex cover of $G$. By a well-known theorem of Terai \cite{T}, $\projdim(S/I)$ is equal to the regularity of the Alexander dual of $I$. Since the regularity of the dual ideal cannot be less than the maximum degree of its minimal generators, the assertion follows.

We claim that $\depth(S/I^{[k]})=2k-1$ for $2\leq k\leq \nu(G)$. Indeed,  if $a,b,c,d$ are pairwise distinct vertices  of $P_n$,  then we may assume  that $\{a,b\}$ and  $\{c,d\}$ are non-edges of $P_n$. Then  $\{a,b\}$ and  $\{c,d\}$ are  edges of $G$. 
This implies that $I^{[2]}=\mm^{[4]}$. Then $I^{[k]}=\mm^{[2k]}$ for all $k\geq 2$ follows from Proposition~\ref{powers equal}. Then, Example~\ref{ramsau}(a) yields the desired conclusion.}
\end{Example}

Let $I\subset S$ be a monomial  ideal with linear quotients. In other words, the elements of $G(I)$ can be ordered as $u_1,\ldots,u_s$ such that for all $i=2,\ldots,s$,  the colon ideal   $(u_1,\ldots,u_{i-1}):u_i$ is generated by variables. Let $r_i$ be the minimum number of variables generating this colon ideal. By \cite[Corollary 8.2.2]{HH} one has
\begin{eqnarray}
\label{lindepth}
\depth (S/I) =n-\max\{r_2,\ldots,r_s\}-1.
\end{eqnarray}

By the \emph{support} of a monomial $u$, denoted by $\mathrm{supp}(u)$, we mean the set of all $i$'s where $x_i$ divides $u$, and we put
$\supp(I)=\cup_{u\in G(I)} \supp (u)$. 

\medskip
In the sequel we will use the following lemma.

\begin{Lemma}
\label{squarefreecolon}
Let $I$ be a squarefree monomial ideal with $G(I)=\{u_1,\ldots,u_s\}$, and let $J=(u_1,\ldots, u_{s-1}):u_s$. If $v\in G(J)$, 
%and suppose that $vu_s\in (u_1,\ldots,u_{s-1})$, 
then $vu_s$ is squarefree. In particular, $v$ is squarefree and $\supp(v)\sect\supp(u_s)=\emptyset$. 
\end{Lemma}

\begin{proof}
Since $v\in J$, it follows that $u_j|vu_s$ for some $j=1,\ldots,s-1$. It follows that $u_j$ divides $w=\prod_{i\in \supp(vu_s)}x_i$ which is a squarefree monomial, because $u_j$ is squarefree. Since $u_s$ is squarefree, we have  $w=v'u_s$  where $v'$ is a squarefree monomial with $v'|v$ and $\supp(v')\sect \supp(u_s)=\emptyset$. Therefore,  $v'\in J$. Since $v\in G(J)$, it follows that $v'=v$. 
\end{proof}

Now we use~(\ref{lindepth}) to show the following theorem for matroidal ideals (i.e. squarefree polymatroidal ideals). See, for example, \cite[Section~12.6]{HHBook} for more properties of polymatroidal ideals.  

\begin{Theorem}
\label{matroid}
Let $I\subset S=K[x_1,\ldots,x_n]$ be a matroidal ideal with  $\supp(I)=[n]$. Then all squarefree powers of $I$ have minimum depth. 
\end{Theorem}

The proof of Theorem~\ref{matroid} will follow  immediately from the next two results. 

\begin{Proposition}
\label{depthmatroid}
Let $I\subset S=K[x_1,\ldots,x_n]$ be a matroidal ideal with  $\supp(I)=[n]$. Suppose that $I$ is generated in degree $d$. Then $\depth (S/I)=d-1$. 
\end{Proposition}

\begin{proof}
By \cite[Lemma 1.3]{HT} we know that $I$ has linear quotients. Let $u_1,\ldots,u_s$ be a linear quotients order for the elements of $G(I)$. By Lemma~\ref{squarefreecolon}, we have $r_i\leq n-d$ for all $i$. Let  $A=\{i\: i\notin \supp(u_s)\}$. Then $|A|=n-d$. Let $i\in A$. Then  there exists $u_j$ such that $x_i|u_j$.   By the exchange property of matroidal ideals, 
%(see \cite[Theorem 12.2]{HH}), 
there exists $k$ such that $x_k|u_s$, $x_k\nmid u_j$ and  $x_i(u_s/x_k)\in I$. This implies that $x_i\in (u_1,\ldots,u_{s-1}):u_s$. Hence, $r_s =n-d$ and (\ref{lindepth}) implies that  $\depth(S/I)= n-(n-d)-1=d-1$, as desired.
\end{proof}

\begin{Proposition}
\label{sqpoly}
The squarefree part of a polymatroidal ideal is a matroidal ideal. 
\end{Proposition}

\begin{proof}
Let $I$ be a polymatroidal ideal with $G(I)=\{u_1, \ldots, u_s\}$ and let $I'$ be the squarefree part of $I$ with $G(I')=\{v_1, \ldots, v_r\}$.  Let $v_i, v_j, x_k$ satisfy $x_k|v_i, x_k \nmid v_j$.  Since $I$ is polymatroidal, there is $x_\ell$ with $x_\ell \nmid v_i, x_\ell | v_j$ for which $x_\ell(v_i/x_k) \in G(I)$.  Since $x_\ell(v_i /x_k)$ is squarefree, it follows that  $x_\ell(v_i /x_k) \in G(I')$.  Thus $I'$ is a matroidal ideal, as desired.
\end{proof}

%{\it Proof of Theorem~\ref{matroid}:}
\begin{proof}({\em Theorem~\ref{matroid}})
By \cite[Theorem 12.6.3]{HH}, $I^k$  is a polymatroidal ideal generated in degree $kd$. Proposition~\ref{sqpoly} implies that $I^{[k]}$ is a matroidal ideal generated in degree $kd$. Thus, Proposition~\ref{depthmatroid} completes the proof. 
\end{proof}

%--------------------------------------------------------
\section{Minimum depth for squarefree powers of edge ideals}\label{Sec:cover criterion}
%--------------------------------------------------------
In this section, we give a characterization of squarefree powers of edge ideals which have minimum depth with respect to reduced homologies of a certain simplicial complex. In particular, we give a more explicit classification of all edge ideals which have minimum depth. We show that all squarefree powers of such ideals have minimum depth as well. Moreover, we give a sufficient condition for squarefree powers of edge ideals in terms of the so-called dominating cliques to have minimum depth. 

Recall that a {\em simplicial complex} $\Delta$ on a finite vertex set $V(\Delta)$ is a set of subsets of $V(\Delta)$ such that $\{v\}\in \Delta$ for every $v\in V(\Delta)$ and if $F\in \Delta$, then $G\in \Delta$ for every $G\subseteq F$. An element $F\in \Delta$ is a {\em face} of $\Delta$ and a {\em facet} is a face of $\Delta$ which is maximal with respect to inclusion. The set of all facets of $\Delta$ is denoted by $\facets(\Delta)$. If $\facets(\Delta)=\{F_1,\dots ,F_q\}$, then we write $\Delta=\langle F_1,\dots ,F_q\rangle$. 

%\medskip
Now, we define a simplicial complex related to matchings of a graph. Let $G$ be a graph and $k=1,\ldots,\nu(G)$. Then, we define 
\[
\Gamma_k(G)=\{F\subseteq V(G): V(M)\not \subseteq F~\text{for any}~k\text{-matching}~M~\text{of}~G\}.
\]
It is easily seen that $\Gamma_k(G)$ is a simplicial complex on $V(G)$ whose Stanley-Reisner ideal is $I(G)^{[k]}$. In particular, $\Gamma_1(G)$ is the well-known independence complex of $G$ (or equivalently the clique complex of the complementary graph $G^c$ of $G$) whose Stanley-Reisner ideal is the edge ideal of $G$.     

\medskip
We know from Proposition~\ref{prop:lower bound} that for any graph $G$, the depth of $S/I(G)^{[k]}$ is at least~$2k-1$. In the following proposition, we give an equivalent condition for attaining this lower bound. 

\begin{Proposition}\label{reduced homology}
	Let $G$ be a graph and $k=1,\ldots,\nu(G)$. Then the following statements are equivalent:
	\begin{enumerate}
		\item[(i)] $g_{I(G)}(k)=0$.
		\item[(ii)] $\tilde{H}_{2k-2}(\Gamma_k(G);K)\neq 0$.
	\end{enumerate}
\end{Proposition}

\begin{proof}
	By the Auslander-Buchsbaum formula, we know that (i) is equivalent to  $\projdim(S/I(G)^{[k]})=n-2k+1$ which holds if and only if $\beta_{n-2k+1,n}(S/I(G)^{[k]})\neq 0$, since $I(G)^{[k]}$ is generated in degree~$2k$. By Hochster's formula, we know that 
	\[
	\beta_{n-2k+1,n}(S/I(G)^{[k]})=\dim_K \tilde{H}_{2k-2}(\Gamma_k(G);K).
	\] 
	Therefore, $\beta_{n-2k+1,n}(S/I(G)^{[k]})\neq 0$ if and only if $\tilde{H}_{2k-2}(\Gamma_k(G);K)\neq 0$, and hence the desired result follows.    
\end{proof} 

The next corollary shows that an edge ideal has minimum depth if and only if the complementary graph of $G$ is disconnected.

\begin{Corollary}\label{prop:depth 1}
	Let $G$ be a graph. Then $\depth(S/I(G))=1$ if and only if $G^c$ is disconnected.
\end{Corollary}

\begin{proof}
	Applying Proposition~\ref{reduced homology} for $k=1$, we have $\depth(S/I(G))=1$ if and only if $\tilde{H}_{0}(\Gamma_1(G);K)\neq 0$ which is equivalent to $\Gamma_1(G)$ being disconnected. The latter is also equivalent to $G^c$ being disconnected, since $\Gamma_1(G)$ is the clique complex of $G^c$. Thus, we get the desired conclusion. 
\end{proof}

Next, we recall some definitions about simplicial complexes and their facet ideals and we fix some notation which will be used in the rest of the section.

A {\em subcollection} of a simplicial complex $\Delta$ is a simplicial complex $\Gamma$ such that every facet of $\Gamma$ is also a facet of $\Delta$. If $A\subseteq V(\Delta)$, then the {\em induced subcollection} $\Delta_A$ is the simplicial complex $\langle F\in \facets(\Delta) \mid F\subseteq A \rangle $.

A set $D\subseteq \facets(\Delta)$ is called a {\em facet cover} of $\Delta$ if every vertex $v$ of $\Delta$ belongs to some $F$ in $D$. A facet cover is called \emph{minimal} if no proper subset of it is a facet cover of $\Delta$.

Let $\Delta$ be a simplicial complex on the vertices $x_1,\ldots, x_n$.
%and let $S=K[x_1,\ldots,x_n]$. 
Recall that the {\em facet ideal} of $\Delta$, denoted by $\mathcal{F}(\Delta)$, is the squarefree monomial ideal 
\[ 
\mathcal{F}(\Delta)=(x_{i_1}\cdots x_{i_k} \mid \{x_{i_1},\ldots, x_{i_k}\}\in \facets(\Delta))\subset S.
\]
If $u$ is a squarefree monomial, then $\Delta_u$ denotes the subcollection $\Delta_U$ where $U$ is the set of variables which divide $u$. 

Erey and Faridi in \cite{EF} introduced the concept of well-ordered facet cover to give a non-vanishingness condition for Betti numbers of facet ideals (see \cite[Definition~3.1]{EF}). Well-ordered facet covers generalize the concept of strongly disjoint bouquets which was introduced by Kimura \cite{K}.

A sequence $F_1,\ldots,F_k$ of facets of a simplicial complex $\Delta$ is called a {\em well-ordered facet cover} if $\{F_1,\ldots,F_k\}$ is a minimal facet cover of $\Delta$ and for every facet $H\notin \{F_1,\ldots,F_k\}$ of $\Delta$ there exists $i\leq k-1$ such that  $F_i \subseteq H \cup F_{i+1} \cup F_{i+2}\cup \cdots \cup F_k.$

Existence of well-ordered facet covers yields non-zero Betti numbers as follows:

\begin{Theorem}\cite[Corollary 3.4]{EF}\label{thm:erey faridi}
	Let $\Delta$ be a simplicial complex and let $u$ be a squarefree
	monomial. If $\Delta_{u}$ has a well-ordered facet cover of cardinality $i$, then $\beta_{i,u}(S/\mathcal{F}(\Delta)) \neq 0$.
\end{Theorem}

We need the following technical lemma which will be useful in the sequel.

\begin{Lemma}\label{lem:choose edges}
	Let $G$ be a graph with no isolated vertices and $\nu(G)\geq 2$. Suppose that $G$ is not a complete bipartite graph and $G^c$ is disconnected, and let $2\leq k\leq \nu(G)$. Then there exists two induced subgraphs $G_1$ and $G_2$ on disjoint sets of vertices, a $(k-1)$-matching $M$ and a vertex $v$ of $G$ such that
	\begin{enumerate}
		\item $V(G_1)\cup V(G_2)=V(G)$,
		\item $\{x_1,x_2\}\in E(G)$ for all $x_1\in V(G_1)$ and $x_2\in V(G_2)$,
		\item $\displaystyle v\notin e$ for every $e\in M$,
		\item $M\cap E(G_1)\neq \emptyset$ and
		\item $v\in V(G_2)$.
	\end{enumerate} 
\end{Lemma}
\begin{proof}
	Since $G^c$ is disconnected, there exist two induced subgraphs $G_1$ and $G_2$ of $G$, on disjoint sets of vertices, satisfying conditions~$(1)$ and~$(2)$. By the symmetry, it suffices to consider the following cases:
	
	Case 1: Suppose that there exists a matching $\{e_1,\dots ,e_k\}$ of $G$ such that $e_1\in E(G_1)$. Let $v$ be a vertex of $G_2$. If there exists $t\in \{2,\dots ,k-1\}$ with $v\in e_t$, then let $M=\{e_1,\dots ,e_k\}\setminus \{e_t\}$. Otherwise, let $M=\{e_1,\dots ,e_{k-1}\}$. In both cases, $v$ and $M$ fulfill the desired conditions.  
	
	Case 2: Suppose that no $k$-matching of $G$ has an edge contained in $G_1$ or $G_2$. Let $\{e_1,\dots ,e_k\}$ be a matching of $G$ with $e_i=\{a_i,b_i\}$ such that $a_i\in V(G_1)$ and $b_i\in V(G_2)$. Without loss of generality, it is enough to consider the following cases:
	
	Case 2.1: Suppose that there is an edge $e$ of $G$ such that $e\subseteq \{a_1,\dots ,a_k\}$. We may assume that $e=\{a_{k-1},a_k\}$. Then $v=b_k$ and $M=\{e_1,\dots ,e_{k-2},e\}$ satisfy the required conditions.
	
    Case 2.2: Suppose that both $A=\{a_1,\dots ,a_k\}$ and $B=\{b_1,\dots ,b_k\}$ are independent sets of $G$. Since $G$ is not a complete bipartite graph, we may assume that there exists an edge $f=\{f_1,f_2\}$ such that $f\in E(G_1)$. Since $A$ is independent, we know that at least one vertex of $f$, say $f_1$, is not in $A$. Without loss of generality, we assume that $e_i\cap f=\emptyset$ for all $i=2,\dots ,k$. Observe that for $v=b_k$ and $M=\{f, e_2,\dots ,e_{k-1}\}$, the required conditions are satisfied.
\end{proof}

In the next theorem, we give a sufficient condition for existence of a well-ordered facet cover of certain cardinality for the simplicial complex whose facet ideal is $I(G)^{[k]}$ where $G$ is assumed to have certain properties.

\begin{Theorem}\label{thm: disconnected complement facet cover}
	Let $G$ be a graph with no isolated vertices and $\nu(G)\geq 2$, which is not a complete bipartite graph. Suppose that $G^c$ is disconnected and for any $2\leq k\leq \nu(G)$, let $\Delta$ be the simplicial complex with facet ideal $I(G)^{[k]}$. Then $\Delta$ has a well-ordered facet cover of cardinality $n-2k+1$.
\end{Theorem}

\begin{proof}
	Let $k=1,\ldots, \nu(G)$ and let $G_1, G_2,M$ and $v$ be as in Lemma~\ref{lem:choose edges}. Suppose that $M=\{e_1,\dots,e_{k-1}\}$ such that $e_1\in E(G_1)$. Let $e_1=\{y,z\}$. We set $U=e_1\cup \dots \cup e_{k-1}\cup \{v\}$. Let $V(G_1)\setminus U=\{x_1,\dots ,x_i\}$ and $V(G_2)\setminus U=\{x_{i+1},\dots ,x_{n-2k+1}\}$. We define \[F_j=e_1\cup \dots \cup e_{k-1}\cup \{v,x_j\}\] for each $j=1,\dots ,n-2k+1$.
	
	If $x_j\in V(G_1)\setminus U$, then $F_j$ is a facet of $\Delta$ corresponding to the $k$-matching $\{\{v,x_j\}\} \cup M$. On the other hand, if $x_j\in V(G_2)\setminus U$, then $F_j$ is a facet corresponding to the $k$-matching $\{\{v,y\}, \{z,x_j\}, e_2,\dots, e_{k-1} \}$. 
	
	We claim that $F_1,\dots ,F_{n-2k+1}$ is a well-ordered facet cover of $\Delta$. Since every vertex of $\Delta$ belongs to some $F_j$, these facets indeed form a cover. Also, since for every $j=1,\dots ,n-2k+1$ we have $x_j\in F_t$ if and only if $j=t$, it follows that this cover is minimal. To prove the ``well-ordered" property, let $H$ be a facet of $\Delta$ such that $H\notin \{F_1,\dots ,F_{n-2k+1}\}$. Observe that $H\not\subseteq U$, since $H$ has $2k$ vertices whereas $U$ has $2k-1$. On the other hand, since $H\notin \{F_1,\dots ,F_{n-2k+1}\}$,
	there exists at least two indices $\ell< m$ such that $\{x_\ell,x_m\}\subseteq H$. Then 
	the proof follows from the inclusion $F_\ell \subseteq H \cup F_{\ell+1}\cup \dots \cup F_{n-2k+1}$.
\end{proof}

Next, we show that having minimum depth for the edge ideal itself implies the same for all squarefree powers and vice versa. 
 
\begin{Corollary}\label{nice}
	Let $G$ be a graph with no isolated vertices and matching number $\nu(G)$. Then the following statements are equivalent:
	\begin{enumerate}
		\item[(i)] $G^c$ is disconnected.
		\item[(ii)] $g_{I(G)}(1)=0$.
		\item[(iii)] $g_{I(G)}(k)=0$ for all $1\leq k\leq \nu(G)$.
	\end{enumerate}
\end{Corollary}
\begin{proof}
	Equivalence of $(i)$ and $(ii)$ was already proved in Corollary~\ref{prop:depth 1}. It is enough to show that $(i)$ implies $(iii).$ Let $2\leq k\leq \nu(G)$. If $G$ is a complete bipartite graph, then the result follows from Example~\ref{ramsau}~(c). Otherwise, let $\Delta$ be the simplicial complex whose facet ideal is $I(G)^{[k]}$. Then by Theorem~\ref{thm: disconnected complement facet cover}, $\Delta$ has a well-ordered facet cover of cardinality $n-2k+1$. Therefore, Theorem~\ref{thm:erey faridi} implies $\beta_{n-2k+1,n}(S/I(G)^{[k]})\neq 0$. Thus, $\projdim (S/I(G)^{[k]})\geq n-2k+1$, and hence the desired result follows from the Auslander-Buchsbaum formula.
\end{proof}

Let $G$ be a graph. 
%with the vertex set $V(G)$. 
A subset $D$ of $V(G)$ is called a \emph{dominating set} if every vertex $v\in V(G)-D$ is adjacent to a vertex in $D$. A complete subgraph $K_m$ of $G$ is called a \emph{dominating clique} if $V(K_m)$ is a dominating set.

The final result of this section applies the notion of dominating cliques to  provide a sufficient condition for having minimum depth for a given squarefree power~$k$.  

\begin{Theorem}\label{thm: dominating clique}
	Let $G$ be a graph and let $2\leq k\leq \nu(G)$. If $G$ has a dominating clique $K_{2k-1}$, then $g_{I(G)}(k)=0$.
\end{Theorem}

\begin{proof}
	Let $K_{2k-1}$ be a dominating clique with $V(K_{2k-1})=\{x_1,\dots ,x_{2k-1}\}$. Let $V(G) - V(K_{2k-1}) =\{x_{2k},x_{2k+1},\dots ,x_n\}$. For every $j=2k,\dots ,n$, we set $F_j=\{x_1,\dots,x_{2k-1}\}\cup \{x_j\}$. Let $\Delta$ be the simplicial complex with facet ideal $\mathcal{F}(\Delta)=I(G)^{[k]}$. Then, each $F_j$ is a facet of $\Delta$. Indeed, without loss of generality, if $x_j$ is adjacent to $x_{2k-1}$, then $\{\{x_1,x_2\},\dots ,\{x_{2k-3},x_{2k-2}\}, \{x_j,x_{2k-1}\}\}$ is a $k$-matching of $G$. It is clear that $\{F_{2k},\dots ,F_{n}\}$ is a minimal facet cover of $\Delta$. As in the proof of Theorem~\ref{thm: disconnected complement facet cover} one can show that $F_{2k},\dots ,F_n$ is a well-ordered facet cover. The proof then follows from Theorem~\ref{thm:erey faridi} and the Auslander-Buchsbaum formula.
\end{proof}

\section{depth of squarefree powers with Linear quotients}\label{Sec:Linear quotients and minimal depth}

In this section, we provide a criterion for squarefree powers of edge ideals with linear quotients to have minimum depth. Applying that criterion, we discuss the depth of squarefree powers of the edge ideal of a class of chordal graphs which are obtained by adding some whiskers to a complete graph. Indeed, we determine when the depth of such ideals is minimum.  

%Let $I$ be a monomial ideal. Then we denote the squarefree part of $I$, i.e. the ideal generated by squarefree generators of $I$, by $I^{sq}$. 
In the next lemma, we show that having linear quotients is inherited by the squarefree part. 

\begin{Lemma}\label{lem:linear quotients squarefree}
	Let $I$ be a monomial ideal with linear quotients. Then the squarefree part of $I$ has also linear quotients.	
\end{Lemma}

\begin{proof}
	Let $G(I)=\{u_1,\dots ,u_s\}$, and  assume that $u_1,\dots ,u_s$ is a  linear quotients ordering. Let $I'$ be the squarefree part of $I$ and  $G(I')=\{u_{i_1},\dots ,u_{i_t}\}$ with $1\leq i_1< \dots <i_t\leq s$. Let $A_j=(u_{i_1},\dots,u_{i_{j-1}}):u_{i_j}$. We show that $A_j$ is generated by variables. Let $v\in G(A_j)$. Then $v\in (u_1,\dots, u_{i_j-1}):u_{i_j}$. Hence there exists $t$ such that $x_t|v$ and  $u_k|x_tu_{i_j}$ for some $k\leq i_j-1$. Lemma~\ref{squarefreecolon} implies that $x_iu_{i_j}$  is squarefree. This implies that $u_k$ is squarefree, and hence $k\in \{i_1,\ldots,i_{j-1}\}$. Therefore, $x_t\in A_j$ which implies that $v=x_t$, since $v\in G(A_j)$ and $x_t|v$. Thus $A_j$ is generated by variables,  as desired.  
\end{proof}

%Suppose on the contrary that $v$ is a minimal generator of $A_j$ which is not a variable. Then there exists $\ell=1,\dots, j-1$ such that $u_{i_\ell}$ divides $vu_{i_j}$. Let $B_j= (u_1,\dots, u_{i_j-1}):u_{i_j}$. We know that $B_j$ is generated by variables, say $B_j=\langle x_{\alpha_1},\dots ,x_{\alpha_s} \rangle$. Since $u_{i_\ell}|vu_{i_j}$ we have $v\in B_j$. So, without loss of generality we assume that $x_{\alpha_1}|v$. Since $v$ is not a variable and it is a minimal generator, it follows that $x_{\alpha_1}\notin A_j$. Therefore there exists a non-squarefree generator of $\langle u_1,\dots ,u_{i_j-1} \rangle$ say $u_p$ such that $u_p|x_{\alpha_1}u_{i_j}$. Then $x_{\alpha_1}|u_{i_j}$ because $u_{i_j}$ is squarefree but $u_p$ is not.

%Since $x_{\alpha_1}|v$ we may assume that $v=x_{\alpha_1}v'$ where $v'$ is a squarefree monomial. Then because $u_{i_\ell}|vu_{i_j}$ it follows that $u_{i_\ell}|v'u_{i_j}$ since $u_{i_\ell}$ is squarefree. Hence $v'\in A_j$ which is a contradiction.

Recall that a \emph{cochordal} graph is a graph whose complementary graph is chordal, i.e. has no induced cycle of length greater than~3. The following is a consequence of Lemma~\ref{lem:linear quotients squarefree}, \cite[Theorem~10.1.9]{HHBook} and \cite[Theorem~10.2.5]{HHBook}.

\begin{Corollary}\label{cochordal linear quotients}
	Let $G$ be a cochordal graph. Then $I(G)^{[k]}$ has linear quotients for any $k=1,\ldots,\nu(G)$. 
\end{Corollary}

In Corollary~\ref{prop:depth 1}, we provided an explicit combinatorial characterization of edge ideals with the minimum depth. In the next proposition, we provide a combinatorial criterion for all squarefree powers of edge ideals with linear quotients admitting minimum depth. 

We call a $k$-matching $M$ a \emph{dominating $k$-matching} if $V(M)$ is a dominating set, i.e. any vertex $v\in V(G)-V(M)$ is adjacent to a vertex in $V(M)$.      

\begin{Proposition}\label{cochordal mindepth}
	Let $G$ be a graph with no isolated vertices and $1\leq k\leq \nu(G)$. If $I(G)^{[k]}$ has linear quotients with respect to the ordering $u_1,\ldots,u_s$ of its minimal generators, then the following statements are equivalent:
	\begin{enumerate}
		\item[(i)] $g_{I(G)}(k)=0$.
		\item[(ii)] There exist a dominating $k$-matching $M$ and some $i=2,\ldots,s$ which satisfy the following conditions:
		\begin{enumerate}
			\item[(a)] $V(M)=\mathrm{supp}(u_i)$, and 
			\item[(b)] for any $t\in V(G)-V(M)$, there exists a $k$-matching $M'$ with $V(M')=\mathrm{supp}(u_j)$ for some $j=1,\ldots,i-1$ such that $V(M')\subseteq V(M)\cup \{t\}$.
		\end{enumerate}  
	\end{enumerate}
	In particular, if $G$ is a cochordal graph, then statements (i) and (ii) are equivalent. 
\end{Proposition}	

\begin{proof}
	Suppose that $V(G)=[n]$. Following notation of Section~\ref{Sec:general squarefree powres}, for any $i=2,\ldots,r$ let $r_i$ be the number of variables in $(u_1,\ldots,u_{i-1}):u_i$. According to~(\ref{lindepth}), $\depth(S/I(G)^{[k]})=2k-1$ if and only if $r_i=n-2k$ for some $i=2,\ldots,r$. This is the case if and only if $(u_1,\ldots,u_{i-1}):u_i$ is generated by $n-2k$ variables, namely 
	\[
	(u_1,\ldots,u_{i-1}):u_i=(x_t: t\in V(G)-\mathrm{supp}(u_i)),
	\] 
	since $|\mathrm{supp}(u_i)|=2k$.
	%It is well-known by the Auslander-Buchsbaum formula that $\depth(S/I(G)^{[k]})=2k-1$ if and only if $\projdim(S/I(G)^{[k]})=n-2k+1$ or equivalently $\projdim(I(G)^{[k]})=n-2k$. Let $a_i$ be the number of variables in $(u_1,\ldots,u_{i-1}):u_i$ for any $i=2,\ldots,r$. By \cite{}, we have $\projdim(I(G)^{[k]})=\max\{a_2,\ldots,a_r\}$. 
	%Let $a_i=\max\{a_2,\ldots,a_r\}$. 
	%Therefore, $\projdim(I(G)^{[k]})=n-2k$ if and only if 
	%$a_i=n-2k$, or equivalently 
	%$(u_1,\ldots,u_{i-1}):u_i$ is generated by $n-2k$ variables, namely 
	%\[
	%(u_1,\ldots,u_{i-1}):u_i=(x_t: t\in V(G)-\mathrm{supp}(u_i)),
	%\] 
	%since $|\mathrm{supp}(u_i)|=2k$. 
	In other words, for any $t\in V(G)-\mathrm{supp}(u_i)$, there exists $j=1,\ldots,i-1$ such that $u_j|x_t u_i$, or equivalently $\mathrm{supp}(u_j)\subseteq \{t\}\cup \mathrm{supp}(u_i)$. 
	%Finally, according to the correspondence between the vertex sets of the $k$-matchings of $G$ and the minimal generators $u_1,\ldots,u_r$ of $I(G)^{[k]}$, it is enough to 
	Let $M$ be a $k$-matching with $V(M)=\mathrm{supp}(u_i)$ and $M'$ be a $k$-matching with $V(M')=\mathrm{supp}(u_j)$. 
	The inclusion $V(M')\subseteq V(M)\cup \{t\}$ for any $t\in V(G)-V(M)$ implies that $t$ is adjacent to some vertices of $M$, and hence $M$ is a dominating $k$-matching of $G$. Thus, the statements (i) and (ii) are equivalent, as desired.  
	
	In particular, if $G$ is cochordal, then the result follows from Corollary~\ref{cochordal linear quotients}.      
\end{proof}
As it was mentioned at the end of the proof of Proposition~\ref{cochordal mindepth}, the condition that $M$ is a ``dominating" matching follows from other conditions. Therefore, we can drop the word ``dominating" from the statement of the proposition. However, to emphasize on this combinatorial condition, we keep it in the statement, especially that this provides us a nice necessary condition for having minimum depth. Indeed, as an immediate consequence of Proposition~\ref{cochordal mindepth}, we get the following necessary condition to have minimum depth.

\begin{Corollary}\label{necessary condition}
	Let $G$ be a graph with no isolated vertices and $k$ be an integer with $1\leq k\leq \nu(G)$ where $I(G)^{[k]}$ has linear quotients. If $g_{I(G)}(k)=0$, then $G$ has a dominating $k$-matching. In particular, the statement holds for any cochordal graph $G$ with no isolated vertices.  	
\end{Corollary} 

As another consequence of Proposition~\ref{cochordal mindepth}, we show that the highest non-zero squarefree power of the edge ideal of any graph has the minimum depth.

\begin{Corollary}\label{nu}
	Let $G$ be a graph with no isolated vertices. Then $g_{I(G)}(\nu(G))=0$. 
\end{Corollary}

\begin{proof}
	Let $k=\nu(G)$. It was proved in \cite[Theorem~5.1]{BHZ} that $I(G)^{[k]}$ has linear quotients with respect to lexicographic order on the generators where the vertices can have any fixed labelling. Let $u_1,\dots,u_s$ be a linear quotients order on the minimal monomial generators of $I(G)^{[k]}$. 
	Let $M$ be a $k$-matching with $V(M)=\mathrm{supp}(u_s)$. It suffices to show that $M$ is a $k$-matching which satisfies condition (ii) of Proposition~\ref{cochordal mindepth}. Let $v\in V(G)-V(M)$.  Since $G$ has no isolated vertices, $v$ is adjacent to at least one vertex of $G$, say $w$. Observe that if $w\notin V(M)$, then $M$ together with the edge $\{v,w\}$ is a matching of $G$ of size greater than $k$, which is a contradiction. Therefore, we may assume that $w\in e$ for some $e\in M$. Then it suffices to put $M'=(M-\{e\})\cup \{e'\}$ with $e'=\{v,w\}$, which completes the proof.
\end{proof}
The corollary above does not generalize to squarefree monomial ideals. Indeed, there are examples of squarefree monomial ideals whose highest non-zero squarefree powers do not have minimum depth. For example, the monomial ideal
\[I=(x_1x_3x_5, x_2x_4x_6, x_5x_7x_9, x_4x_6x_8, x_4x_7x_{10}, x_9x_{10}x_{11}, x_5x_8x_{11})\]
in $S=K[x_1,\ldots,x_{11}]$ satisfies $\nu(I)=3$ with $g_I(3)=1\neq 0$.

\medskip
As an application of Proposition~\ref{cochordal mindepth}, we discuss the depth of squarefree powers of edge ideals of \emph{multiple whiskered complete graphs} which are obtained by attaching some whiskers to each vertex of a complete graph $K_s$ with $s\geq 2$. We denote such a graph by $H(a_1,\ldots,a_s)$ where $a_i\geq 0$ is the number of whiskers attached to the vertex $i$ of $K_s$. Here the vertex set of $K_s$ is assumed to be $[s]=\{1,\ldots,s\}$. 
%Figure~\ref{Fig:multiwhiskered} depicts the graph $H(1,2,1,5,3)$. 
The case $G=H(a_1,\ldots,a_s)$ with $a_1=\cdots=a_s=s-1$ came up in \cite{Ha-Hi} where edge ideals of minimum projective dimension were considered. In the same article, besides other results,  it was shown that $\projdim(S/I(G))=s-2$ which means that $\depth (S/I(G))=s^2-2s+2$ is not minimum.   

\begin{Theorem}\label{multiwhiskered depth}
	Let $G=H(a_1,\ldots,a_s)$ with $a_i\geq 1$ for all $i=1,\ldots,s$ and let  $k=1,\ldots,\nu(G)$. Then we have:
	\begin{enumerate}
		\item[(a)] $\nu(G)=s$;
		\item[(b)] $G$ is cochordal;
		\item[(c)] The following statements are equivalent: 
		\begin{enumerate}
			\item[(i)] $g_{I(G)}(k)=0$; 
			\item[(ii)] $\lfloor s/2 \rfloor+1 \leq k \leq s$. 
		\end{enumerate}
	\end{enumerate}
\end{Theorem}	

\begin{proof}
	(a) and (b) can be easily proved. We prove (c). By Corollary~\ref{cochordal linear quotients}, $I(G)^{[k]}$ has linear quotients. Let $u_1,\dots ,u_r$ be a linear quotients ordering for the minimal generators of $I(G)^{[k]}$. 
	
	First suppose that $s=2k$.  Assume on the contrary that there is a dominating $k$-matching which satisfies condition (ii) of Proposition~\ref{cochordal mindepth}. Let $u_m$ be the generator corresponding to such matching. Since $s=2k$, we have $u_m=x_1\dots x_s$. Let $W$ be the set of all leaves of $G$. By assumption, for every $v\in W$, there exists $j<m$ such that $u_j | u_mx_v$. 
	For each $v\in W$, let $u_{\bar{v}}$ be the smallest generator in the linear quotients ordering whose support contains $v$ but no other leaves. In other words, we define
	\begin{equation}\label{eq: min index}
	\bar{v}=\min\{t: \mathrm{supp}(u_t)\cap W =\{v\}\}.
	\end{equation}
	Note that we have $\bar{v}<m$ for every leaf $v\in W$ by the initial assumption on $u_m$. We set
	$\alpha=\max\{\bar{v} : v\in W\}$. The support of $u_\alpha$ has exactly one leaf, say $w$. Without loss of generality, we assume that $w$ is adjacent to $1$. Since $u_\alpha$ is of degree $2k$, there is exactly one vertex of $K_s$ that is missing in the support of $u_\alpha$, say $j\neq 1$. Let $z$ be a leaf adjacent to $j$. By definition of $\alpha$, we have $\bar{z}<\alpha <m$. Observe that $u_{\bar{z}}/\gcd(u_{\bar{z}},u_\alpha)=x_jx_z$. Now we consider the ideal
	\[
	J=(u_1,\dots ,u_{\alpha-1}):u_\alpha.
	\]
	 Since $J$ is generated by variables, either $x_j$ or $x_z$ must be a generator of $J$. Since $z$ is a leaf, $x_z$ cannot be a generator of $J$. Then there exists $\beta <\alpha$ such that $u_{\beta}/\gcd(u_{\beta},u_\alpha)=x_j$. If the support of $u_\beta$ has no leaves, then $\beta=m$ which is a contradiction as $\beta <\alpha <m$. On the other hand, since the support of $u_\beta$ cannot contain multiple leaves, $w$ must be the only leaf in it. In that case, $\beta < \alpha =\bar{w}$ contradicts the definition \eqref{eq: min index} of $\bar{w}$.

	Next, suppose that $s\neq 2k$. Let $W_i=\{v_1^{(i)},v_2^{(i)},\ldots,v_{a_i}^{(i)}\}$ denote the set of  leaves attached to the vertex $i$ of $K_s$. 
	
	\medskip 
	(i)\implies (ii): If either $s$ is odd with $1\leq k\leq \lfloor s/2 \rfloor$ or $s$ is even with $1\leq k < s/2$, then for each $k$-matching $M$ of $G$ one has $[s]\not \subseteq V(M)$. 
	Since each $a_i\geq 1$, it follows that $M$ can not be a dominating $k$-matching of $G$. Hence, $\depth(S/I(G)^{[k]})>2k-1$, by Corollary~\ref{necessary condition}. 
	
	\medskip 
	(ii)\implies (i): If $k=s$, then the result follows from Corollary~\ref{nu}. Assume that $\lfloor s/2 \rfloor+1\leq k < s$. Then there is a $k$-matching $N$ of $G$ with $[s]\subsetneq V(N)$. Let $M_1,\ldots,M_q$ be those $k$-matchings of $G$ on distinct sets of vertices with $[s]\subsetneq V(M_t)$ for each $t$. 
	%Let $u_1,\ldots,u_r$ denote the ordering of the minimal generators of $I(G)^{[k]}$ which induces linear quotients. 
	We may assume that $V(M_i)=\mathrm{supp}(u_{{\ell}_i})$ with $1\leq {\ell}_1<\cdots <{\ell}_q\leq r$. Letting $M=M_q$, we claim that $M$ is a $k$-matching which satisfies condition (ii) of Proposition~\ref{cochordal mindepth}. To this end, let $v\in V(G)-V(M)$, say $v=v_j^{(i)}$. Now we consider the following two cases to conclude the proof.
	
	First suppose that $a_i\geq 2$ and $\{i, v_{j'}^{(i)}\}\in M$ where $j\neq j'$. Then 
	\[
	M'=(M-\{\{i, v_{j'}^{(i)}\}\})\cup \{\{i, v_j^{(i)}\}\}
	\] 
	is a $k$-matching of $G$ and $M'=M_{q'}$ with $q'<q$ and $V(M')\subseteq V(M)\cup \{v\}$.  
	
	Next suppose that $\{i,v_{j'}^{(i)}\}\notin M$ for all $1\leq j'\leq a_i$. Then there is some $i'\neq i$ with $1\leq i' \leq s$ such that $\{i,i'\}\in M$. Also, there is some $i''\neq i,i'$ with $1\leq i'' \leq s$ such that $\{i'',v_{j''}^{(i'')}\}\in M$ for some $j''$. Then  
	\[
	M''=(M-\{\{i,i'\},\{i'', v_{j''}^{(i'')}\}\})\cup \{\{i',i''\},\{i, v_j^{(i)}\}\}
	\] 
	is  a $k$-matching of $G$ and $M''=M_{q''}$ with $q''<q$ and $V(M'')\subseteq V(M)\cup \{v\}$.
\end{proof}

Besides the given characterization in Theorem~\ref{multiwhiskered depth}, it would be also interesting to find the exact values of the normalized depth function for $1 \leq k \leq \lfloor s/2 \rfloor$. In the next example, we give a few computed cases.  
  
\begin{Example}\label{computation}
	{\em 
Our computations with CoCoA \cite{CoCoA} shows the following for the squarefree powers of edge ideals of multiple whiskered complete graphs $G$ with $a_i=1$ for all $i=1,\ldots,s$ which do not have minimum depth:
\begin{enumerate}
	\item If $s=4$, then $g_{I(G)}(1)=3$ and $g_{I(G)}(2)=1$.
	\item If $s=5$, then $g_{I(G)}(1)=4$ and $g_{I(G)}(2)=2$.
	\item If $s=6$, then $g_{I(G)}(1)=5$, $g_{I(G)}(2)=3$ and $g_{I(G)}(3)=1$.
\end{enumerate}
}
\end{Example}

\medskip
Finally, we would like to remark that in the proof of Theorem~\ref{multiwhiskered depth}, we used the fact that $a_i\geq 1$, for all $i=1,\ldots,s$. If we also allow some $a_i$'s to be equal to zero, then, using Theorem~\ref{thm: dominating clique}, we can guarantee for a specific squarefree power to have minimum depth. More precisely, let $G$ be a multiple whiskered graph with $a_i=0$ for at least one $i$ and let $k\geq 2$. If either $s=2k$ or $s=2k-1$, then $G$ has a dominating $K_{2k-1}$ clique. Therefore, by Theorem~\ref{thm: dominating clique}, it follows that $S/I(G)^{[k]}$ has minimum depth.

\section*{Acknowledgment}

J\"urgen Herzog and Sara Saeedi Madani was supported by T\"UB\.ITAK (2221-Fellowships for Visiting Scientists and Scientists on Sabbatical Leave) to visit Nursel Erey at Gebze Technical University. Takayuki Hibi was partially supported by JSPS KAKENHI 19H00637. Sara Saeedi Madani was in part supported by a grant from IPM (No. 1401130112).

\end{document}